\documentclass[12pt,a4paper]{article}
\usepackage{}

\usepackage{latexsym}
\usepackage{amsmath}
\usepackage{amssymb}
\usepackage{arydshln}

\usepackage{cite}
\usepackage{stmaryrd}
\usepackage{enumerate}

\usepackage{hyperref}

\usepackage{color}
\usepackage{lineno}
\usepackage{graphicx}
\usepackage{ae}
\usepackage{amsmath}
\usepackage{amssymb}
\usepackage{latexsym}
\usepackage{url}
\usepackage{epsfig}
\usepackage{mathrsfs}
\usepackage{amsfonts}
\usepackage{amsthm}

\usepackage{float}
\usepackage{subfig}
\usepackage{tikz}
\usetikzlibrary{positioning}

\newtheorem{theorem}{Theorem}

\newtheorem{lemma}{Lemma}

\theoremstyle{definition}

\newtheorem{claim}{Claim}

\newtheorem{conjecture}{Conjecture}
\newtheorem{problem}{Problem}
\newtheorem{case}{Case}
\newtheorem{subcase}{Case}[case]

\numberwithin{equation}{section} 
\allowdisplaybreaks    

\long\def\delete#1{}

\usepackage{xcolor}
\usepackage[normalem]{ulem}

\begin{document}
\title{A note on two cycles of consecutive even lengths in graphs}
\author{
Binlong Li\footnote{School of Mathematics and Statistics, Northwestern Polytechnical University, Xi'an, Shaanxi 710072, P. R. China.
Supported by NSFC (No. 12071370) and Shaanxi Fundamental Science Research Project for Mathematics and Physics (No. 22JSZ009). Email: binlongli@nwpu.edu.cn.},
Yufeng Pan\footnote{School of Mathematics and Statistics, Northwestern Polytechnical University, Xi'an, Shaanxi 710072, P. R. China.
Email: yf.pan@mail.nwpu.edu.cn.},
Lingjuan Shi\footnote{School of Software, Northwestern Polytechnical University, Xi'an, Shaanxi 710072, P. R. China.
Supported by NSFC (No. 11901458) and the Natural Science Foundation of Shaanxi Province (No. 2024JC-YBQN-0053). Email: shilj18@nwpu.edu.cn.}
}
\date{}

\openup 0.5\jot
\maketitle

\begin{abstract}
Bondy and Vince proved that a graph of minimum degree at least three contains two cycles whose lengths differ by one or two, which was conjectured by Erd\H{o}s. Gao, Li, Ma and Xie gave an average degree counterpart of Bondy-Vince's result, stating that every $n$-vertex graph with at least $\frac{5}{2}(n-1)$ edges contains two cycles of consecutive even lengths, unless $4|(n-1)$ and every block of $G$ is a clique $K_5$. This confirms the case $k=2$ of Verstra\"ete's conjecture, which states that every $n$-vertex graph without $k$ cycles of consecutive even lengths has edge number $e(G)\leq\frac{1}{2}(2k+1)(n-1)$, with equality if and only if every block of $G$ is a clique of order $2k+1$. Sudakov and Verstra\"{e}te further conjectured that if $G$ is a graph with maximum number of edges that does not contain $k$ cycles of consecutive even lengths, then every block of $G$ is a clique of order at most $2k+1$. In this paper, we prove the case $k=2$ for Sudakov-Verstra\"{e}te's conjecture, by extending the results of Gao, Li, Ma and Xie.
\end{abstract}

\section{Introduction}
The research of cycles has been fundamental since the beginning of graph theory. One of various
problems on cycles which have been considered is the study of the distribution of cycle length.
Erd\H{o}s posted the question that whether every graph of minimum degree at least three contains two cycles whose lengths differ by one or two. Using a structural argument based on a lemma on induced cycles of Thomassen and
Toft \cite{Thomas-Toft}, Bondy and Vince \cite{Bondy98} answered the question
affirmatively by proving the following stronger theorem.
\begin{theorem}[Bondy and Vince~\cite{Bondy98}]\label{Bondy-Vince}
 With the exception of $K_1$ and $K_2$, every graph having at most two vertices of degree less than three contains two cycles of lengths differing by one or two.
\end{theorem}

For this result, it is natural to generalize to consider the conditions for the existence of $k$ cycles of consecutive lengths.
Under the conditions of average degree, Verstra\"ete \cite{J.V16} proposed the following conjecture.

\begin{conjecture}[Verstra\"ete~\cite{J.V16}]
 If $G$ is an $n$-vertex graph not containing $k$ cycles of consecutive even lengths, then $e(G) \leq \frac{1}{2}(2k+1)(n-1)$, with equality if and only if every block of $G$ is a clique of order $2k + 1$.
\end{conjecture}

In 1993, Dean et.al. \cite{Dean93} proved that the case of $k=1$ is true, i.e., any graph with $n$ vertices and at least $\frac{3}{2}(n-1)$ edges contains an even cycle, unless every block of the graph is a triangle. Recently, Gao, Li, Ma and Xie \cite{Gao23} proved the case $k=2$ of this conjecture in the following result.

\begin{theorem}[Gao et al.~\cite{Gao23}]\label{Li-Ma}
Let $G$ be a graph of order $n$ with $e(G) \geq \frac{5}{2}(n-1)$. Then $G$ contains two cycles of consecutive even lengths, unless $4|(n-1)$ and every block of $G$ is a clique $K_5$.
\end{theorem}

In terms of the minimum degree, Thomassen \cite{Thomas} conjectured that every graph of minimum degree at least $k+1$ contains cycles of all possible even lengths modular $k$. This was proved by Gao, Huo, Liu and Ma \cite{A-unified-proof}.
Correspondingly, Sudakov and Verstra\"ete \cite{B.S17} proposed
the following strengthening of Thomassen's conjecture \cite{Thomas}.

\begin{conjecture}[Sudakov and Verstra\"ete~\cite{B.S17}]\label{conjecture-Sudakov}
 If $G$ is a graph with maximum number of edges that does not contain $k$ cycles of consecutive even lengths, then every block of $G$ is a clique of order at most $2k+1$.
\end{conjecture}

In this paper, we confirm the conjecture \ref{conjecture-Sudakov} for case $k=2$. The specific conclusion is as follows.
\begin{theorem}\label{main}
Let $G$ be a graph of order $n$ with $e(G)\geq 10q +\binom{r+1}{2}$, where $n-1=4q+r$, $0\leq r<4$. Then G contains two cycles of consecutive even lengths, unless\\
(i) $r=0$, and $G$ is a connected graph consisting of $q$ blocks isomorphic to $K_{5}$; or\\
(ii) $r\neq 0$, and $G$ is a connected graph consisting of $q$ blocks isomorphic to $K_{5}$ and $1$ block isomorphic to $K_{r+1}$.
\end{theorem}

\section{Some preliminaries}

In this and the next sections, we always assume that $n-1=4q+r$, $0\leq r<4$. We define $\mathcal{G}_n$ to be the class of $n$-vertex graphs as describing in Theorem \ref{main}. That is, for every graph $G\in\mathcal{G}_n$, $G$ is an $n$-vertex connected graph such that\\
(i) if $r=0$, then $G$ consists of $q$ blocks isomorphic to $K_{5}$; \\
(ii) if $r\neq 0$, then $G$ consists of $q$ blocks isomorphic to $K_{5}$ and $1$ block isomorphic to $K_{r+1}$.
For the later case, we will call the $K_{r+1}$-block the \emph{remainder block} of $G$.

\begin{lemma}\label{4-6-cycle}
Each of the following graphs $G$ contains a 4-cycle and a 6-cycle:\\
(i) $G$ is obtained from a clique $K:=K_4$ by adding two vertices $x,y$ such that $xy\in E(G)$, and each of $x,y$ have two neighbors in $K$;\\
(ii) $G$ is obtained from a clique $K:=K_4$ by adding two vertices $x,y$ such that $x$ has three neighbors in $K$ and $y$ has two neighbors in $V(K)\cup\{x\}$;\\
(iii) $G$ is obtained from a triangle $K:=K_3$ by adding four vertices $x_1,y_1,x_2,y_2$ such that $x_1y_1,x_2y_2\in E(G)$, and each of $x_1,y_1,x_2,y_2$ have two neighbors in $K$.
\end{lemma}

\begin{proof}
  (i) Clearly $K$ contains a 4-cycle. Set $V(K)=\{v_1,v_2,v_3,v_4\}$. By condition there are two nonadjacent edges from $x,y$ to $K$. Without loss of generality we can assume that $xv_1,yv_2\in E(G)$. Now $v_1xyv_2v_3v_4v_1$ is a 6-cycle.

  (ii) Clearly $K$ contains a 4-cycle. Set $V(K)=\{v_1,v_2,v_3,v_4\}$. If $xy\notin E(G)$, then $x,y$ have a common neighbor in $K$, say $v_1$. Without loss of generality we can assume that $v_2$ is a neighbor of $y$ other than $v_1$, and $v_3$ is a neighbor of $x$ other than $v_1,v_2$. Now $v_1yv_2v_4v_3xv_1$ is a 6-cycle. Now suppose that $xy\in E(G)$. Without loss of generality we can assume that $v_1$ is a neighbor of $y$ other than $x$, and $v_2$ is a neighbor of $x$ in $K$ other than $v_1$. Now $v_1yxv_2v_3v_4v_1$ is a 6-cycle.

  (iii) We notice that $G[V(K)\cup\{x_1\}]$ contains a 4-cycle since $x_1$ has two neighbors in the triangle $K$. Set $V(K)=\{v_1,v_2,v_3\}$. By condition $x_1,x_2$ have a common neighbor in $K$, say $v_1$. If $y_1,y_2$ have a common neighbor other than $v_1$, say $y_1v_2,y_2v_2\in E(G)$, then $v_1x_1y_1v_2y_2x_2v_1$ is a 6-cycle. Now assume that $y_1,y_2$ have no common neighbors other than $v_1$. Without loss of generality we can assume that $y_1$ neighbors $v_1,v_2$ and $y_2$ neighbors $v_1,v_3$. It follows that $v_1x_1y_1v_2v_3y_2v_1$ is a 6-cycle.
\end{proof}

\begin{lemma}\label{lemma-main}
  Let $G\in\mathcal{G}_n$.\\
  (i) If $G'$ is a graph obtained from $G$ by adding one extra vertex $x$ and two edges $xu_1,xu_2$, where $u_1,u_2\in V(G)$, then either $G'$ contains two cycles of consecutive even lengths or $G'$ is a spanning subgraph of a graph in $\mathcal{G}_{n+1}$.\\
  (ii) If $G'$ is a graph obtained from $G$ by adding two extra vertices $x,y$ and five edges $xy,xu_1,xu_2,yv_1,yv_2$, where $u_1,u_2,v_1,v_2\in V(G)$ (possibly $u_i=v_j$), then either $G'$ contains two cycles of consecutive even lengths or $G'$ is a spanning subgraph of a graph in $\mathcal{G}_{n+2}$.
\end{lemma}

\begin{proof}
We let $B_0$ be the remainder block $K_{r+1}$ of $G$ (when $r\neq 0$). We first prove the following claim.

\begin{claim}
  Let $u,v$ be two distinct vertices of $G$. Then $G$ contains four paths from $u$ to $v$ with consecutive lengths, unless both $u,v\in V(B_0)$.
\end{claim}

\begin{proof}
  Suppose that either $u$ or $v$ is not in $B_0$. If $u,v$ are contained in a common block $B$ of $G$, then $B$ is a $K_5$. It follows that $B$ contains four paths from $u$ to $v$ of lengths 1, 2, 3, 4, respectively, as desired. Now assume that $u,v$ are not contained in a common block. Without loss of generality we can assume that $u\notin V(B_0)$. Then there is a vertex $w$ separating $u,v$ in $G$ such that $u,w$ are contained in a common block, say $B$. Now $B$ contains four paths from $u$ to $w$ of lengths 1, 2, 3, 4. Together with a common path from $w$ to $v$, we can find four paths from $u$ to $v$ with consecutive lengths. $\hfill\square$
\end{proof}

Now we prove the lemma.

(i) If one of $u_1,u_2$ is not in $V(B_0)$, then $G$ contains four $(u_1,u_2)$-paths of consecutive lengths. Together with $u_1xu_2$, we can find four cycles of consecutive lengths in $G'$, two of which have consecutive even lengths. Now we assume that both $u_1,u_2\in V(B_0)$. It follows that $r\neq 0$, and $B_0$ is a $K_2, K_3$ or $K_4$. The graph $G''$ obtained from $G$ by adding $x$ and adding all edges $xu$ with $u\in V(B_0)$ is a supergraph of $G'$ and $G''\in\mathcal{G}_{n+1}$, as desired.

(ii) By (i), we can assume that $u_1,u_2,v_1,v_2\in V(B_0)$. If $r=1$ or 2, then $B_0$ is a $K_2$ or $K_3$. The graph $G''$ obtained from $G$ by adding $x,y$ and adding all edges $xy,xu,yu$ with $u\in V(B_0)$ is a supergraph of $G'$ and $G''\in\mathcal{G}_{n+2}$, as desired. Now assume that $r=3$ and $B_0$ is a $K_4$. By Lemma \ref{4-6-cycle} (i), $G'$ has a $4$-cycle and a $6$-cycle. This proves the assertion. $\hfill\blacksquare$
\end{proof}

Gao et al. \cite{Gao23} present the proof of Theorem $2$ by reducing to the following Theorems \ref{Gao-LI-MA} and \ref{3-connected}.
Afterwards, we call a path from vertex $x$ to vertex $y$ an $(x,y)$-path.
\begin{theorem}[Gao et al. \cite{Gao23}]\label{Gao-LI-MA}
Let $G$ be a graph with $x, y \in V (G)$ such that $G + xy $ is $2$-connected. If every vertex of $G$ other than $x$ and $y$ has degree at least $3$, and any edge $uv \in E(G)$ with  ${\{u, v}\} \cap {\{x, y}\} =\emptyset$ has degree sum $d_G(u)+ d_G(v) \geq 7$, then there exist two $(x,y)$-paths in $G-xy$ whose lengths differ by two.
\end{theorem}

\begin{theorem}[Gao et al. \cite{Gao23}]\label{3-connected}
Every $3$-connected graph of order at least $6$ contains two cycles of consecutive even lengths.
\end{theorem}

It is known that for a $k$-connected graph $G$, and two vertex subsets $X$ and $Y$ of $G$ with cardinalities both at least $k$, $G$ has a family of $k$ pairwise disjoint $(X, Y)$-paths (here, an $(X, Y)$-path is a path from one vertex in $X$ to one vertex in $Y$).
Applying this conclusion, the following Lemma can be easily obtained. However, for the sake of the following description, we list it as the following.

\begin{lemma}\label{2-path}
If $G+xy$ is a $2$-connected graph and $G-xy$ is not bipartite, then there exists an odd length $(x,y)$-path and an even length $(x,y)$-path in $G-xy$.
\end{lemma}

The following fact was used in the proof of Theorem \ref{Li-Ma} in \cite{Gao23}. We will use it repeatedly in the next section. So it is convenient to list it as a lemma.

\begin{lemma}\label{cutxy}
  Let $G$ be a 2-connected graph, $\{x,y\}$ be a 2-cut of $G$, and $H_1$ be a component of $G-\{x,y\}$. If every vertex of $G$ other than $x$ and $y$ has degree at least 3, and any edge $uv\in E(H_1)$ satisfies $d(u)+d(v)\geq 7$, then $G$ contains two cycles of consecutive even lengths.
\end{lemma}

\begin{proof}
  Let $H_2$ be a component of $G-\{x,y\}$ other than $H_1$, and let $G_i=G[V(H_i)\cup\{x,y\}]$, $i=1,2$. Notice that every vertex in $G_2$ other that $x,y$ has degree at least 3. By Theorem \ref{Bondy-Vince}, $G_2$ contains two cycles of lengths differing by one or two. If $G_2$ is bipartite, then $G_2$ contains two cycles of consecutive even length, as desired. Now we assume that $G_2$ is not bipartite. By Lemma \ref{2-path}, $G_2$ contains two $(x,y)$-paths of lengths with different parity.

  By Theorem \ref{Gao-LI-MA}, $G_1$ contains two $(x,y)$-paths $P_1,P_2$ of lengths differing 2. Let $P_3$ be an $(x,y)$-path in $G_2$ of length the same parity as $P_1$. It follows that $P_1\cup P_3$ and $P_2\cup P_3$ are two cycles in $G$ of consecutive even lengths. $\hfill\blacksquare$
\end{proof}

\section{Proof of Theorem 3}
We use induction on $n$. For $1\leq n\leq5$, by the assumption of $e(G)$, $G$ is a complete graph of order $n$. Hence Theorem \ref{main} holds.
So in the following, we suppose $n\geq 6$.
Now suppose that $G$ is a graph of order $n$ and edge number $e(G)\geq 10q +\binom{r+1}{2}$ such that $G$ does not contain two cycles of consecutive even lengths.
We will show that $G\in \mathcal{G}_n$. If not, then we have the following Claim 1.
\setcounter{claim}{0}
\begin{claim}\label{kappa2}
 $\kappa(G)=2$.
\end{claim}

\begin{proof}
By the Theorem \ref{3-connected} and the assumption of $G$, $G$ is not $3$-connected.
Now, we suppose that $G$ is not $2$-connected. Then $G$ is the union of two nontrivial graphs $G_1, G_2$ which share a vertex $x$. Let $|V(G_i)|=n_i$ and $n_i-1=4q_i+r_i$, $0\leq r_i\leq 3$, $i=1,2$. Then $n=n_1+n_2-1$ and either $q=q_1+q_2$, $r=r_1+r_2$, or $q=q_1+q_2+1$, $r=r_1+r_2-4$. Note that $G_i$, $i=1,2$, does not contain two cycles of consecutive even lengths. By induction hypothesis, we see that $e(G_i)\leq 10q_i+{r_i+1\choose 2}$. It follows that
$$e(G_1)+e(G_2)\leq 10q_1+{r_1+1\choose 2}+10q_2+{r_2+1\choose 2}\leq 10q+{r+1\choose 2}.$$
Moreover, the first equality holds if and only if $G_i\in\mathcal{G}_{n_i}$, $i=1,2$; and the second equality holds if and only if either $r_1$ or $r_2=0$. Notice that $e(G)=e(G_1)+e(G_2)\geq10q+{r+1\choose 2}$, we see that both equalities hold in above inequality. That is, $G_1\in\mathcal{G}_{n_1}$, $G_2\in\mathcal{G}_{n_2}$ and either $r_1$ or $r_2=0$. It follows that $G\in\mathcal{G}_n$.  $\hfill\blacksquare$
\end{proof}

By Claim \ref{kappa2}, $G$ has a $2$-cut ${\{x, y}\}$, and $G$ is the union of two graphs $G_1,G_2$, common to $x,y$. Since $G$ is $2$-connected, each $G_i+xy$ is $2$-connected.
We now divide the proof into four cases according to $r$.

\begin{case}
$r=0$.
\end{case}

In this case $4|n-1$ and the assertion can be deduced from Theorem \ref{Li-Ma}, immediately.

\begin{case}
$r=1$.
\end{case}

In this case, we have $e(G)\geq 10q +1$.
By Claim \ref{kappa2}, any vertex of $G$ has degree at least $2$. Next, we divide the proof into the following two subcases according to whether $G$ has some vertex of degree $2$.

\begin{subcase}
$G$ has a vertex of degree $2$.
\end{subcase}

Let $u\in V(G)$ be such that $d(u)=2$, and let $G'=G-u$. We first show that $\delta(G')\geq 4$. Suppose that there is a vertex $v\in V(G')$ with $d_{G'}(v)\leq 3$. Then the graph $G''=G-\{u,v\}$ has $n-2$ vertices and has edge number
$$e(G'')\geq e(G)-5=10q+1-5=10(q-1)+6,$$
implying that $G''\in\mathcal{G}_{n-2}$ and $d_{G'}(v)=3$. Let $B_0$ be the remainder block of $G''$. By Lemma \ref{lemma-main}, $G'\in\mathcal{G}_{n-1}$, implying that $v$ has no neighbors in $V(G'')\backslash V(B_0)$. We claim that $u$ has no neighbors in $V(G'')\backslash V(B_0)$ as well. Suppose otherwise that $uu_1\in E(G)$ with $u_1\in V(G'')\backslash V(B_0)$. Let $u_2$ be a second neighbor of $u$ in $G'$. Then $G'$ contains four $(u_1,u_2)$-paths of consecutive lengths. Together with $u_1uu_2$, we find four cycles of $G$ of consecutive lengths, and two of which are of consecutive even lengths, a contradiction. Now $B_0$ is a $K_4$, $v$ has three neighbors in $B_0$ and $u$ has two neighbors in $V(B_0)\cup\{v\}$. By Lemma \ref{4-6-cycle} (ii), $G[V(B_0)\cup\{u,v\}]$ contains a 4-cycle and a 6-cycle, a contradiction. Thus as we claimed, $\delta(G')\geq 4$.

If $G'$ is 3-connected, then either $G'=K_5$ or $G'$ has order at least 6. By Lemma \ref{lemma-main} for the former and by Theorem \ref{3-connected} for the later, we have that $G$ contains two cycles of consecutive even lengths, a contradiction. If $\kappa(G')=2$, then $u\notin\{x,y\}$, and $\{x,y\}$ is a 2-cut of $G'$. By Lemma \ref{cutxy}, $G'$ contains two cycles of consecutive even lengths. If $G'$ has a cut-vertex, then $u$ is contained in a 2-cut of $G$, say $\{u,v\}$. By Lemma \ref{cutxy}, $G$ contains two cycles of consecutive even lengths, both contradictions.

\begin{subcase}
 Every vertex of $G$ has degree at least $3$.
\end{subcase}

We first show that any edge $uv\in E(G)$ satisfies $d(u)+d(v)\geq 7$. By the contrary, we suppose that there exists an edge $uv\in E(G)$ such that $d(u)+d(v)=6$. Then the subgraph $G-\{u,v\}$ has $n-2$ vertices and at least $10q-4=10(q-1)+6$ edges.
By induction hypothesis, either $G-\{u,v\}$ contains two cycles of consecutive even lengths, or $G-\{u,v\}\in\mathcal{G}_{n-2}$.
So by the assumption of $G$, the latter case holds. By Lemma \ref{lemma-main} (ii), $G$ contains two cycles of consecutive even lengths or $G\in\mathcal{G}_n$, as desired. Thus as we claimed, for any edge $uv\in E(G)$, $d(u)+d(v)\geq 7$. Now by Lemma \ref{cutxy}, $G$ contains two cycles of consecutive even lengths, a contradiction.

\begin{case}
$r=2$.
\end{case}

In this case $e(G)\geq 10q+3$. We first show that $\delta(G)\geq3$. Suppose not, then $G$ has a vertex $v$ with $d(v)=2$ (recall that $G$ is 2-connected). It follows that $G-v$ has $n-1$ vertices and at least $10q+1$ edges. Since $G-v$ does not contain two cycles of consecutive even lengths, by induction hypothesis, $G-v\in\mathcal{G}_{n-1}$. By Lemma \ref{lemma-main} (i), either $G$ contains two cycles of consecutive even lengths or $G$ is a subgraph of a member in $\mathcal{G}_n$. For the later case we must have $G\in\mathcal{G}_n$ by the edge number of $G$, as desired. Thus as we claimed, $\delta(G)\geq3$.

Recall that $G$ is the union of two graphs $G_1,G_2$, common to $x,y$. If for some $i=1,2$, every edge $uv\in E(G_i)$ with $\{u,v\}\cap\{x,y\}=\emptyset$ satisfies $d(u)+d(v)\geq 7$, then by Lemma \ref{cutxy}, $G$ contains two cycles of consecutive even lengths, a contradiction. This implies that there is an edge $u_iv_i\in E(G_i)$ with $\{u_i,v_i\}\cap\{x,y\}=\emptyset$ such that $d(u_i)+d(v_i)=6$, $i=1,2$.

Let $G'=G-\{u_1,v_1,u_2,v_2\}$ and let $G'_i=G[V(G')\cup\{u_i,v_i\}]$, $i=1,2$. Then $G'$ has $n-4$ vertices and $e(G)-10=10(q-1)+3$ edges. By induction hypothesis, $G'\in\mathcal{G}_{n-4}$. Let $B_0$ be the remainder block of $G'$. Then $B_0$ is a triangle. Notice that $u_iv_i\in E(G)$ and both $u_i,v_i$ have two neighbors in $G'$. By Lemma \ref{lemma-main} (ii), $G'_i$ is a subgraph of a member in $\mathcal{G}_{n-2}$. This implies that all neighbors of $u_i, v_i$ in $G'$ are in $B_0$. 
So $G[V(B_0)\cup\{u_1, v_1, u_2, v_2\}]$ is the graph as depicted in Lemma \ref{4-6-cycle} (iii), and 
it contains a $4$-cycle and a $6$-cycle, a contradiction.

\begin{case}
$r=3$.
\end{case}

In this case $e(G)\geq 10q+6$. We claim that $\delta(G)\geq 4$.
Suppose not, then $G$ has a vertex $v$ with $d(v)\leq 3$.
Then the subgraph $G-v$ has $n-1$ vertices and at least $10q+3$ edges.
By induction hypothesis, either $G-v$ contains two cycles of consecutive even lengths, or $G-v\in\mathcal{G}_{n-1}$. Since $G$ does not contain two cycles of consecutive even lengths, the latter case holds. By Lemma \ref{lemma-main} (i), $G$ either contains two cycles of consecutive even lengths or $G$ is a subgraph of a member in $\mathcal{G}_n$. For the later case we must have $G\in\mathcal{G}_n$ by the edge number of $G$, as desired. Thus as we claimed, $\delta(G)\geq 4$. It follows that for any edge $uv\in E(G)$, $d(u)+d(v)\geq 7$. By Lemma \ref{cutxy}, $G$ contains two cycles of consecutive even lengths, a contradiction.

All the above contradictions imply that Theorem \ref{main} holds. $\hfill\blacksquare$

\section{Concluding remarks}

If $\mathcal{H}$ is any family of graphs, then $ex(n, \mathcal{H})$ denotes the maximum number of edges
in an $n$-vertex $\mathcal{H}$-free graph. These quantities are collectively referred to as the
Tur\'{a}n numbers for $\mathcal{H}$.
For $k>\ell\geq0$, let $\mathcal{C}_{\ell\bmod k}$ denote the family of all
cycles of length $\ell$ mod $k$.

In this paper, we got a further development of conclusions of paper \cite{Gao23}, that is we showed that for a graph $G$ of order $n$ if $e(G)\geq 10q +\binom{r+1}{2}$, then $G$ contains two cycles of consecutive even lengths, unless $G\in \mathcal{G}_n$. Notice that two cycles of consecutive even lengths always contains one of length $0\bmod 4$ and one of length $2\bmod 4$. On the other hand, the graphs in $\mathcal{G}_n$ contain no cycles of length $2\bmod 4$. Thus we have:
\newcommand{\ex}{\mathrm{ex}}
$$\ex(n, \mathcal{C}_{2\bmod 4})=10q+\binom{r+1}{2}, \mbox{ where } n-1=4q+r, 0\leq r<4.$$
Moreover, a graph $G$ is extremal for $ex(n, \mathcal{C}_{2\bmod 4})$ if and only if $G\in\mathcal{G}_n$.

Bollob\'{a}s \cite{bollobas} was the first to show that $\ex(n, \mathcal{C}_{\ell\bmod k})$ is linear in $n$ whenever $\mathcal{C}_{\ell\bmod k}$ contains some even cycles, and showed $\ex(n, \mathcal{C}_{\ell\bmod k})\leq\frac{1}{k}[(k+1)^k-1]n$. This upper bound was reduced by a number of authors, see \cite{Bondy98,chen-saito, Dean93, Diwan2010, Fan2002} (just to mention a few).
When we fix $\ell=0$, for the existence of cycles of length $0\bmod k$ in graphs, a lot of questions have also been raised in some papers.
Dean et. al \cite{Dean93} researched some such questions for the case $k=4$ in terms of the minimum degree.
Recently, Gy\H{o}ri et. al \cite{EG-Li-2023} precisely determined the maximum number of edges in a graph
containing no cycle of length $0\bmod 4$, and got the Tur\'{a}n number for $\mathcal{C}_{0\bmod 4}$, which is,
$$\ex(n, \mathcal{C}_{0\bmod 4})=\left\lfloor\frac{19}{12}(n-1)\right\rfloor.$$
For any $n\geq2$, they presented a graph of order $n$ attaining this upper bound.
However, the corresponding extremal graphs have not been completely characterized.

For the Tur\'{a}n number of $\mathcal{C}_{3\bmod 4}$, we have
$$\ex(n,\mathcal{C}_{3\bmod 4})=\ex(n,C_3)=\left\lfloor\frac{n^2}{4}\right\rfloor.$$
Since $C_3\in\mathcal{C}_{3\bmod 4}$, we have that $\ex(n,\mathcal{C}_{3\bmod 4})\leq\ex(n,C_3)$. On the other hand, the extremal graphs for $\ex(n,C_3)$, which are the nearly balanced complete bipartite graph $K_{\lceil\frac{n}{2}\rceil,\lfloor\frac{n}{2}\rfloor}$, contain no cycles of length $3\bmod 4$. It follows that $\ex(n,\mathcal{C}_{3\bmod 4})=\ex(n,C_3)$.

For the Tur\'{a}n number of $\mathcal{C}_{1\bmod 4}$, we have
$$\ex(n, \mathcal{C}_{1\bmod 4})=\ex(n,C_5)=\left\{
\begin{array}{ll}
\binom{n}{2}, & n\leq 4; \\
7, & n=5; \\
\lfloor\frac{n^2}{4}\rfloor, & n\geq6.
\end{array}\right.$$
Since $C_5\in\mathcal{C}_{1\bmod 4}$, we have that $\ex(n,\mathcal{C}_{1\bmod 4})\leq\ex(n,C_5)$. On the other hand, the extremal graphs for $\ex(n,C_5)$ (which are $K_n$ for $n\leq 4$, $K_1\vee(K_1\cup K_3)$ or $3K_1\vee K_2$ for $n=5$, $K_{3,3}$ or $K_1\vee(K_2\cup K_3)$ or $4K_1\vee K_2$ for $n=6$, $K_{3,4}$ or $K_1\vee 2K_3$ for $n=7$
 and $K_{\lceil\frac{n}{2}\rceil,\lfloor\frac{n}{2}\rfloor}$ for $n\geq 8$, see \cite{odd-cycle}) contain no cycles of length $1\bmod 4$. It follows that $\ex(n,\mathcal{C}_{1\bmod 4})=\ex(n,C_5)$.

To finish the section, we pose the following question:
\begin{problem}
  Describe all extremal graphs for $\mathcal{C}_{0\bmod 4}$.
\end{problem}

\end{document}